\documentclass[12pt]{article}

\usepackage{amsmath, epsfig, cite}
\usepackage{amssymb}
\usepackage{amsfonts}
\usepackage{latexsym}
\usepackage{amsthm}

\newtheorem{thm}{Theorem}[section]

\newtheorem{cor}[thm]{Corollary}

\newtheorem{lem}[thm]{Lemma}

\numberwithin{equation}{section}

\renewcommand{\thefootnote}{}

\begin{document}

\begin{center}
{\large\bf Some Dwork-type  $q$-supercongruences from a $_6\phi_5$
summation formula
 \footnote{ Corresponding author$^*$. Email addresses: weichuanan78@163.com (C. Wei), hmuqwang@163.com (Q. Wang)}}
\end{center}

\renewcommand{\thefootnote}{$\dagger$}

\vskip 2mm \centerline{Chuanan Wei, Qin Wang$^*$}
\begin{center}
{School of Intelligent Medicine and Technology\\ Hainan Medical
University, Haikou 571199, China}
\end{center}


\vskip 0.7cm \noindent{\bf Abstract.} With the help of a $_6\phi_5$
summation formula and Guo and Zudilin's method, we shall establish
some Dwork-type  $q$-supercongruences in this paper. When $q\to1$,
these $q$-supercongruences are able to engender the corresponding
supercongruences. One of them may be stated as follows: for any
prime $p\geq5$   and any positive integer $s$,
\begin{align*}
&\sum_{k=0}^{p^s-1}(6k-1)\frac{(-\frac{1}{3})_k^3}{(1)_k^3} \equiv
0\pmod{p^{3s}}.
\end{align*}



\section{Introduction}
For a complex variable $x$ and a nonnegative integer $n$, define the
shifted-factorial by
\[(x)_{0}=1\quad \text{and}\quad (x)_{n}
=x(x+1)\cdots(x+n-1)\quad \text{when}\quad n\in\mathbb{Z}^{+}.\] In
1996, Van Hamme \cite[(H.2)]{Hamme} discovered the following nice
 supercongruence:
\begin{equation}
\sum_{k=0}^{(p-1)/2}\frac{(\frac{1}{2})_k^3}{k!^3}\equiv
\begin{cases} \displaystyle -\Gamma_p(\tfrac{1}{4})^4  \pmod{p^2}, &\text{if $p\equiv 1\pmod 4$,}\\[2mm]
 0\pmod{p^2}, &\text{if $p\equiv 3\pmod 4$.}
\end{cases}
\label{eq:hamme}
\end{equation}
Here and all over the paper, $p$ is an odd prime and $\Gamma_p(x)$
stands for the $p$-adic Gamma function. After several years, Long
and Ramakrishna \cite{LR} proved the generalization of
\eqref{eq:hamme}:
\begin{equation}
\sum_{k=0}^{(p-1)/2}\frac{(\frac{1}{2})_k^3}{k!^3}\equiv
\begin{cases} \displaystyle -\Gamma_p(\tfrac{1}{4})^4  \pmod{p^3}, &\text{if $p\equiv 1\pmod 4$,}\\[2mm]
 \displaystyle -\frac{p^2}{16}\Gamma_p(\tfrac{1}{4})^4\pmod{p^3}, &\text{if $p\equiv 3\pmod 4$.}
\end{cases}
\label{eq:long-a}
\end{equation}
 In the same paper, they gave also
 the interesting supercongruence:
\begin{equation}\label{eq:long-b}
\sum_{k=0}^{p-1}\frac{(\frac{1}{3})_k^3}{(1)_k^3}\equiv
\begin{cases} \displaystyle \Gamma_p\big(\tfrac{1}{3}\big)^6  \pmod{p^3}, &\text{if $p\equiv 1\pmod 6$,}\\[10pt]
\displaystyle
-\frac{p^2}{3}\Gamma_p\big(\tfrac{1}{3}\big)^6\pmod{p^3}, &\text{if
$p\equiv 5\pmod 6$.}
\end{cases}
\end{equation}
In 2022, the first author \cite{Wei-a} proposed the following
conjecture: for any prime $p$  and any positive integer $s$ such
that $p^s\equiv 1\pmod 6$,
\begin{align}\label{eq:wei}
&\sum_{k=0}^{(p^s+2)/3}\frac{(-\frac{2}{3})_k^3}{(1)_k^3} \equiv
0\pmod{p^{3s}}.
\end{align}
Hu and Wang \cite{Hu} confirmed this conjecture and provided the
following result: for any prime $p$  and any positive integer $s$
such that $p^s\equiv 1\pmod 6$,
\begin{align}\label{eq:hu}
&\sum_{k=0}^{p^s-1}\frac{(-\frac{2}{3})_k^3}{(1)_k^3} \equiv
0\pmod{p^{3s}}.
\end{align}
Subsequently, Guo \cite{Guo-2026} offered the supplementary formula:
for any prime $p$  and any positive integer $s$ such that $p^s\equiv
5\pmod 6$,
\begin{align}\label{eq:guo}
&\sum_{k=0}^{M}\frac{(-\frac{2}{3})_k^3}{(1)_k^3} \equiv
0\pmod{p^{3s}},
\end{align}
where $M=(2p^s+2)/3$ or $p^s-1$.

 For two complex numbers $x$ and $q$ and a nonnegative integer $n$, define the $q$-shifted factorial
 to be
 \begin{equation*}
(x;q)_{\infty}=\prod_{k=0}^{\infty}(1-xq^k)\quad\text{and}\quad
(x;q)_n=\frac{(x;q)_{\infty}}{(xq^n;q)_{\infty}}\quad
\text{when}\quad n\in\mathbb{Z}^{+}\cup\{0\}.
 \end{equation*}
For the aim of simplicity, we usually adopt the compact notation:
\begin{equation*}
(x_1,x_2,\dots,x_m;q)_{n}=(x_1;q)_{n}(x_2;q)_{n}\cdots(x_m;q)_{n},
 \end{equation*}
where $m\in\mathbb{Z}^{+}$ and $n\in\mathbb{Z}^{+}\cup\{0,\infty\}.$

In recent years, the research of $q$-supercongruences attracts a lot
of people. In 2019, Guo and Zudilin \cite{GuoZua} established the
following $q$-analogue of \eqref{eq:hamme}: for any positive odd
integer $n$,
\begin{align*}
&\sum_{k=0}^{(n-1)/2}\frac{(q;q^2)_k^2(q^2;q^4)_k}{(q^2;q^2)_k^2(q^4;q^4)_k}q^{2k}
\notag\\[5pt]
&\equiv
\begin{cases} \displaystyle\frac{(q^2;q^4)_{(n-1)/4}^2}{(q^4;q^4)_{(n-1)/4}^2}q^{(n-1)/2}\pmod{\Phi_n(q)^2}, &\text{if $n\equiv 1\pmod 4$,}
\\[2mm]
 \displaystyle 0\pmod{\Phi_n(q)^2}, &\text{if $n\equiv 3\pmod 4$.}
\end{cases}
\end{align*}
Here and throughout the paper,  $[r]$ denotes the $q$-integer
$(1-q^r)/(1-q)$ and $\Phi_n(q)$ reorensts the $n$-th cyclotomic
polynomial in $q$:
\begin{equation*}
\Phi_n(q)=\prod_{\substack{1\leqslant k\leqslant n\\
\gcd(k,n)=1}}(q-\zeta^k),
\end{equation*}
where $\zeta$ is an $n$-th primitive root of unity. Some years
later, Guo \cite[Theorem 1]{Guo-new} and the first author
\cite[Theorem 1.1]{Wei} provided the following generalization of Guo
and Zudilin's result, which is also a $q$-analogue of
\eqref{eq:long-a}: for any positive odd integer $n$,
\begin{align*}
&\sum_{k=0}^{(n-1)/2}\frac{(q;q^2)_k^2(q^2;q^4)_k}{(q^2;q^2)_k^2(q^4;q^4)_k}q^{2k}
\notag\\
 &\quad\equiv
\begin{cases} \displaystyle \frac{(q^2;q^4)_{(n-1)/4}^2}{(q^4;q^4)_{(n-1)/4}^2}q^{(n-1)/2}
A_q(n)\pmod{\Phi_n(q)^3}, &\text{if $n\equiv 1\pmod 4$,}\\[2mm]
 \displaystyle \frac{(q^3;q^4)_{(n-1)/2}}{(q^5;q^4)_{(n-1)/2}}[n]\pmod{\Phi_n(q)^3}, &\text{if $n\equiv 3\pmod 4$,}
\end{cases}
\end{align*}
where
$$A_q(n)=1+2[n]^2\sum_{i=1}^{(n-1)/4}\frac{q^{4i-2}}{[4i-2]^2}.$$
Wei, Liu, and Wang \cite{Wei-b} found the following $q$-analogue of
\eqref{eq:long-b}: modulo $\Phi_n(q)^3$,
\begin{equation*}
\sum_{k=0}^{n-1}\frac{(q;q^3)_k^3}{(q^3;q^3)_k^3}q^{3k}\equiv
\begin{cases} \displaystyle q^{(n-1)/3}\frac{(q^2;q^3)_{(n-1)/3}^2}{(q^3;q^3)_{(n-1)/3}^2}A_q(n), &\text{if $n\equiv 1\pmod 3$,}\\[10pt]
\displaystyle
q^{(2n-1)/3}\frac{(q^2;q^3)_{(2n-1)/3}^2}{(q^3;q^3)_{(2n-1)/3}^2}B_q(n),
&\text{if $n\equiv 2\pmod 3$,}
\end{cases}
\end{equation*}
where $n$ is a positive integer and
\begin{align*}
&A_q(n)=1+[n]^2\sum_{i=1}^{(n-1)/3}\frac{q^{3i-1}}{[3i-1]^2},
 \\[2mm]
 &B_q(n)=
1-[2n]^2\sum_{i=1}^{(2n-1)/3}\frac{q^{3i-1}}{[3i-1]^2}.
\end{align*}
A $q$-analogue of \eqref{eq:wei} and a $q$-analogue of \eqref{eq:hu}
can be seen in Hu and Wang \cite{Hu} and a $q$-analogue of
\eqref{eq:guo} can be seen in Guo \cite{Guo-2026}. Stronger
$q$-analogues of \eqref{eq:wei}-\eqref{eq:guo} due to Qin and Wang
\cite{Qin} may be stated as follows:
\begin{align*}
&\sum_{k=0}^{N}\frac{1+q^{3k-1}}{1+q^{-1}}\frac{(q^{-2};q^3)_k^3}{(q^3;q^3)_k^3}q^{6k}
\equiv0\pmod{[n]^3},
\end{align*}
where $n>3$ is an integer coprime with $6$ and $N$ is given by
\begin{equation*}
N=
\begin{cases} \displaystyle (n+2)/3\quad\text{or}\quad n-1, &\text{if $n\equiv 1\pmod 6$,}\\[10pt]
\displaystyle (2n+2)/3\quad\text{or}\quad n-1 , &\text{if $n\equiv
5\pmod 6$.}
\end{cases}
\end{equation*}
 For more $q$-analogues of supercongruences, we refer the reader to
the papers \cite{Guo-adb,Guo-a2,Guo-2025,Guo2023,LW}.

Motivated by the works just mentioned, we shall establish the
following three theorems.

\begin{thm}\label{thm-a}
Let $n\geq5$ be an integer coprime with $6$. Then
\begin{align}\label{eq:wei-a}
&\sum_{k=0}^{n-1}[6k-1]\frac{(q^{-1};q^3)_k^3}{(q^3;q^3)_k^3}q^{3k}
\equiv0\pmod{[n]^3}.
\end{align}
\end{thm}

\begin{thm}\label{thm-b}
Let $n$ be a positive integer such that $n\equiv 5\pmod 6$. Then
\begin{align}\label{eq:wei-b}
&\sum_{k=0}^{(n+1)/3}[6k-1]\frac{(q^{-1};q^3)_k^3}{(q^3;q^3)_k^3}q^{3k}
\equiv0\pmod{[n]^3}.
\end{align}
\end{thm}

\begin{thm}\label{thm-c}
Let $n\geq7$ be an integer such that $n\equiv 1\pmod 6$. Then
\begin{align}\label{eq:wei-c}
&\sum_{k=0}^{(2n+1)/3}[6k-1]\frac{(q^{-1};q^3)_k^3}{(q^3;q^3)_k^3}q^{3k}
\equiv0\pmod{[n]^3}.
\end{align}
\end{thm}

Setting $n=p^s$ and taking $q\to 1$ in Theorems
\ref{thm-a}-\ref{thm-c}, we catch hold of the following conclusions.

\begin{cor}\label{cor-a}
Let $p\geq5$ be a prime and $s$ a positive integer. Then
\begin{align*}
&\sum_{k=0}^{p^s-1}(6k-1)\frac{(-\frac{1}{3})_k^3}{(1)_k^3} \equiv
0\pmod{p^{3s}}.
\end{align*}
\end{cor}

\begin{cor}\label{cor-b}
Let $p$ be a prime and $s$ a positive integer such that $p^s\equiv
5\pmod 6$. Then
\begin{align*}
&\sum_{k=0}^{(p^s+1)/3}(6k-1)\frac{(-\frac{1}{3})_k^3}{(1)_k^3}
\equiv 0\pmod{p^{3s}}.
\end{align*}
\end{cor}

\begin{cor}\label{cor-c}
Let $p$ be a prime and $s$ a positive integer such that  and
$p^s\equiv 1\pmod 6$. Then
\begin{align*}
&\sum_{k=0}^{(2p^s+1)/3}(6k-1)\frac{(-\frac{1}{3})_k^3}{(1)_k^3}
\equiv 0\pmod{p^{3s}}.
\end{align*}
\end{cor}

In terms of a $_6\phi_5$ summation formula and Guo and Zudilin's
method introduced in \cite{GuoZu}, we are going to prove Theorem
\ref{thm-a} in Section 2. Similarly, the proof of Theorems
\ref{thm-b} and \ref{thm-c} will be provided in Sections 3 and 4,
respectively.

\section{Proof of Theorem \ref{thm-a}}

For the aim to prove Theorem \ref{thm-a}, we require the following
two lemmas.

\begin{lem}\label{lemma-a}
Let $n\geq5$ be an integer coprime with $6$ and let $m$ be a
positive integer such that $m|n$ and $m\equiv5\pmod{6}$. Then,
modulo $[n]\prod_{j=0}^{(n-m)/m}(1-aq^{(3j+1)m})(a-q^{(3j+1)m}),$
\begin{align}
&\sum_{k=0}^{n-1}[6k-1]\frac{(q^{-1},aq^{-1},q^{-1}/a;q^3)_k}{(q^3,q^3/a,aq^{3};q^3)_k}q^{3k}
\equiv0. \label{eq:wei-aa}
\end{align}
\end{lem}

\begin{proof}
Recall the nonterminating $_6\phi_5$ summation formula (cf.
\cite[Appendix (II.20)]{Gasper}): Under the convergent condition
$|aq/bcd|<1$,
\begin{align}
_{6}\phi_{5}\!\left[\begin{array}{c}
a,\,qa^{\frac{1}{2}},\, -qa^{\frac{1}{2}},\, b,\, c,\, d \\
a^{\frac{1}{2}},\,-a^{\frac{1}{2}},\, aq/b,\, aq/c,\, aq/d
\end{array};q,\,\frac{aq}{bcd}  \right]
=\frac{(aq,aq/bc,aq/bd,aq/cd;q)_{\infty}}{(aq/b,aq/c,aq/d,aq/bcd;q)_{\infty}},
\label{Dixon}
\end{align}
where the basic hypergeometric series $_{r+1}\phi_{r}$ has been
defined by
$$
_{r+1}\phi_{r}\left[\begin{array}{c}
a_1,a_2,\ldots,a_{r+1}\\
b_1,b_2,\ldots,b_{r}
\end{array};q,\, z
\right] =\sum_{k=0}^{\infty}\frac{(a_1,a_2,\ldots, a_{r+1};q)_k}
{(q,b_1,b_2,\ldots,b_{r};q)_k}z^k.
$$

Performing  the replacements $a\mapsto q^{-1}$, $b\mapsto
q^{-1+(3j+1)m}$, $c\mapsto q^{-1-(3j+1)m}$, $d\mapsto q$, $q\mapsto
q^{3}$ in the identity \eqref{Dixon}, we have
\begin{align}
&\sum_{k=0}^{((3j+1)m+1)/3}\frac{1-q^{6k-1}}{1-q^{-1}}\frac{(q^{-1},q^{-1+(3j+1)m},q^{-1-(3j+1)m};q^3)_k}
{(q^{3},q^{3-(3j+1)m},q^{3+(3j+1)m};q^3)_k}q^{3k}
\notag\\[2mm]
&\quad=\frac{(q^{2},q^4,q^{2+(3j+1)m},q^{2-(3j+1)m};q^3)_{\infty}}{(q,q^3,q^{3+(3j+1)m},q^{3-(3j+1)m};q^3)_{\infty}}
\notag\\[2mm]
&\quad=0.\label{Dixon-a}
\end{align}
It is routine to realize the following two relations:
$$\text{When}\:\: k>((3j+1)m+1)/3,\:\:\text{there is}\:\: (q^{-1-(3j+1)m};q^3)_k=0;$$
$$\text{When}\:\: 0\leq j\leq (n-m)/m,\:\:\text{there holds}\:\: ((3j+1)m+1)/3\leq n-1.$$
Hence we discover the following equation:
\begin{align*}
&\sum_{k=0}^{n-1}\frac{1-q^{6k-1}}{1-q^{-1}}\frac{(q^{-1},q^{-1+(3j+1)m},q^{-1-(3j+1)m};q^3)_k}
{(q^{3},q^{3-(3j+1)m},q^{3+(3j+1)m};q^3)_k}q^{3k}=0,
\end{align*}
where $0\leq j\leq (n-m)/m$. Observing that these polynomials
$(1-aq^{m}),(1-aq^{4m}),\ldots,(1-aq^{3n-2m}),(a-q^{m}),(a-q^{4m}),\ldots,(a-q^{3n-2m})$
 are pairwise coprime, we obtain the following
$q$-supercongruence: modulo
$\prod_{j=0}^{(n-m)/m}(1-aq^{(3j+1)m})(a-q^{(3j+1)m}),$
\begin{align}
&\sum_{k=0}^{n-1}[6k-1]\frac{(q^{-1},aq^{-1},q^{-1}/a;q^3)_k}{(q^3,q^3/a,aq^{3};q^3)_k}q^{3k}
\equiv0. \label{eq:wei-bb}
\end{align}
Guo and Schlosser \cite[Lemma 2.2]{GS20} showed that
\eqref{eq:wei-bb} is valid modulo $[n]$. Since the two polynomials
$\prod_{j=0}^{(n-m)/m}(1-aq^{(3j+1)m})(a-q^{(3j+1)m}))$ and $[n]$
are relatively prime, we arrive at \eqref{eq:wei-aa} to complete the
proof.
\end{proof}

Similar to the proof of Lemma \ref{lemma-a}, we can verify the
following lemma.

\begin{lem}\label{lemma-b}
Let $n\geq5$ be an odd integer coprime with $6$ and let $m>1$ be an
integer such that $m|n$ and $m\equiv1\pmod{6}$. Then, modulo
$[n]\prod_{j=0}^{(n-m)/m}(1-aq^{(3j+2)m})(a-q^{(3j+2)m}),$
\begin{align}
&\sum_{k=0}^{n-1}[6k-1]\frac{(q^{-1},aq^{-1},q^{-1}/a;q^3)_k}{(q^3,q^3/a,aq^{3};q^3)_k}q^{3k}
\equiv0. \label{eq:wei-cc}
\end{align}
\end{lem}

Now we are ready to prove Theorem \ref{thm-a}.

\begin{proof}[Proof of Theorem \ref{thm-a}]
Because $m|n$ and $m\equiv5\pmod{6}$, the $a=1$ case of the
expression $[n]\prod_{j=0}^{(n-m)/m}(1-aq^{(3j+1)m})(a-q^{(3j+1)m})$
has the factor
 $$\prod_{m|n, m\equiv5{\!\!\!\!\!}\pmod{6}}\Phi_m(q)^{(2n+m)/m},$$
 while the $a=1$ case of the expression $(q^3/a,aq^{3};q^3)_{n-1}$,
 which is in the denominator of the left hand side of
 \eqref{eq:wei-aa}, only contains the factor
  $$\prod_{m|n, m\equiv5{\!\!\!\!\!}\pmod{6}}\Phi_m(q)^{(2n-2m)/m}.$$

 Because  $m|n$, $m>1$, and $m\equiv1\pmod{6}$, the $a=1$ case of the
expression $[n]\prod_{j=0}^{(n-m)/m}(1-aq^{(3j+2)m})(a-q^{(3j+2)m})$
has the factor
 $$\prod_{m|n, m>1,m\equiv1{\!\!\!\!\!}\pmod{6}}\Phi_m(q)^{(2n+m)/m},$$
 whereas the $a=1$ case of the expression $(q^3/a,aq^{3};q^3)_{n-1}$,
 which is in the denominator of the left hand side of
 \eqref{eq:wei-cc}, merely contains the factor
  $$\prod_{m|n,m>1, m\equiv1{\!\!\!\!\!}\pmod{6}}\Phi_m(q)^{(2n-2m)/m}.$$
Therefore, \eqref{eq:wei-a} is correct modulo
$$\prod_{m|n, m\equiv5{\!\!\!\!\!}\pmod{6}}\Phi_m(q)^{3}\prod_{m|n, m>1,m\equiv1{\!\!\!\!\!}\pmod{6}}\Phi_m(q)^{3}
=\prod_{m|n, m>1,}\Phi_m(q)^{3}=[n]^3.$$
\end{proof}

\section{Proof of Theorem \ref{thm-b}}

For the purpose of proving Theorem \ref{thm-b}, we need the
following two lemmas.

\begin{lem}\label{lemma-c}
Let $n$ be a positive integer such that $n\equiv 5\pmod 6$ and let
$m$ be a positive integer subject to $m|n$ and $m\equiv5\pmod{6}$.
Then, modulo
$[n]\prod_{j=0}^{(n-m)/3m}(1-aq^{(3j+1)m})(a-q^{(3j+1)m}),$
\begin{align}
&\sum_{k=0}^{(n+1)/3}[6k-1]\frac{(q^{-1},aq^{-1},q^{-1}/a;q^3)_k}{(q^3,q^3/a,aq^{3};q^3)_k}q^{3k}
\equiv0. \label{eq:wei-dd}
\end{align}
\end{lem}

\begin{proof}
It is not difficult to understand the two facts:
$$\text{When}\:\: k>((3j+1)m+1)/3,\:\:\text{there is}\:\: (q^{-1-(3j+1)m};q^3)_k=0;$$
$$\text{When}\:\: 0\leq j\leq (n-m)/3m,\:\:\text{there holds}\:\: ((3j+1)m+1)/3\leq (n+1)/3.$$
According to them and \eqref{Dixon-a}, we find the following
equation:
\begin{align*}
&\sum_{k=0}^{(n+1)/3}\frac{1-q^{6k-1}}{1-q^{-1}}\frac{(q^{-1},q^{-1+(3j+1)m},q^{-1-(3j+1)m};q^3)_k}
{(q^{3},q^{3-(3j+1)m},q^{3+(3j+1)m};q^3)_k}q^{3k}=0,
\end{align*}
where $0\leq j\leq (n-m)/3m$. Noting that these polynomials
$(1-aq^{m}),(1-aq^{4m}),\ldots,(1-aq^{n}),(a-q^{m}),(a-q^{4m}),\ldots,(a-q^{n})$
 are pairwise coprime, we get the following
$q$-supercongruence: modulo
$\prod_{j=0}^{(n-m)/3m}(1-aq^{(3j+1)m})(a-q^{(3j+1)m}),$
\begin{align}
&\sum_{k=0}^{(n+1)/3}[6k-1]\frac{(q^{-1},aq^{-1},q^{-1}/a;q^3)_k}{(q^3,q^3/a,aq^{3};q^3)_k}q^{3k}
\equiv0. \label{eq:wei-ee}
\end{align}
Guo and Schlosser \cite[Lemma 2.2]{GS20} showed that
\eqref{eq:wei-ee} is valid modulo $[n]$. Since the two polynomials
$\prod_{j=0}^{(n-m)/3m}(1-aq^{(3j+1)m})(a-q^{(3j+1)m}))$ and $[n]$
are relatively prime, we are led to \eqref{eq:wei-dd} to complete
the proof.
\end{proof}

Similar to the proof of Lemma \ref{lemma-c}, we can deduce the
following lemma.

\begin{lem}\label{lemma-d}
Let $n$ be a positive integer such that $n\equiv 5\pmod 6$ and let
$m>1$ be an integer such that $m|n$ and $m\equiv1\pmod{6}$. Then,
modulo $[n]\prod_{j=0}^{(n-2m)/3m}(1-aq^{(3j+2)m})(a-q^{(3j+2)m}),$
\begin{align}
&\sum_{k=0}^{(n+1)/3}[6k-1]\frac{(q^{-1},aq^{-1},q^{-1}/a;q^3)_k}{(q^3,q^3/a,aq^{3};q^3)_k}q^{3k}
\equiv0. \label{eq:wei-ff}
\end{align}
\end{lem}

Now we begin to prove Theorem \ref{thm-b}.

\begin{proof}[Proof of Theorem \ref{thm-b}]
On account of $m|n$ and $m\equiv5\pmod{6}$, the $a=1$ case of the
product $[n]\prod_{j=0}^{(n-m)/3m}(1-aq^{(3j+1)m})(a-q^{(3j+1)m})$
owns the factor
 $$\prod_{m|n, m\equiv5{\!\!\!\!\!}\pmod{6}}\Phi_m(q)^{(2n+7m)/3m},$$
 while the $a=1$ case of the $q$-shifted factorials $(q^3/a,aq^{3};q^3)_{(n+1)/3}$,
 which appears in the denominator of the left hand side of
 \eqref{eq:wei-dd}, only implies the factor
  $$\prod_{m|n, m\equiv5{\!\!\!\!\!}\pmod{6}}\Phi_m(q)^{(2n-2m)/3m}.$$

 On account of  $m|n$, $m>1$, and $m\equiv1\pmod{6}$, the $a=1$ case of the
product $[n]\prod_{j=0}^{(n-2m)/3m}(1-aq^{(3j+2)m})(a-q^{(3j+2)m})$
owns the factor
 $$\prod_{m|n, m>1,m\equiv1{\!\!\!\!\!}\pmod{6}}\Phi_m(q)^{(2n+5m)/3m},$$
 whereas the $a=1$ case of the $q$-shifted factorials $(q^3/a,aq^{3};q^3)_{(n+1)/3}$,
 which appears in the denominator of the left hand side of
 \eqref{eq:wei-ff}, merely implies the factor
  $$\prod_{m|n,m>1, m\equiv1{\!\!\!\!\!}\pmod{6}}\Phi_m(q)^{(2n-4m)/3m}.$$
So \eqref{eq:wei-b} is right modulo
$$\prod_{m|n, m\equiv5{\!\!\!\!\!}\pmod{6}}\Phi_m(q)^{3}\prod_{m|n, m>1,m\equiv1{\!\!\!\!\!}\pmod{6}}\Phi_m(q)^{3}
=\prod_{m|n, m>1,}\Phi_m(q)^{3}=[n]^3.$$
\end{proof}

\section{Proof of Theorem \ref{thm-c}}

In order to prove Theorem \ref{thm-c}, we demand the following two
lemmas.

\begin{lem}\label{lemma-e}
Let $n\geq7$ be an integer such that $n\equiv 1\pmod 6$ and let $m$
be a positive integer subject to $m|n$ and $m\equiv5\pmod{6}$. Then,
modulo $[n]\prod_{j=0}^{(2n-m)/3m}(1-aq^{(3j+1)m})(a-q^{(3j+1)m}),$
\begin{align}
&\sum_{k=0}^{(2n+1)/3}[6k-1]\frac{(q^{-1},aq^{-1},q^{-1}/a;q^3)_k}{(q^3,q^3/a,aq^{3};q^3)_k}q^{3k}
\equiv0. \label{eq:wei-gg}
\end{align}
\end{lem}

\begin{proof}
It is ordinary to know the following two formulas:
$$\text{When}\:\: k>((3j+1)m+1)/3,\:\:\text{there is}\:\: (q^{-1-(3j+1)m};q^3)_k=0;$$
$$\text{When}\:\: 0\leq j\leq (2n-m)/3m,\:\:\text{there holds}\:\: ((3j+1)m+1)/3\leq (2n+1)/3.$$
By means of them and \eqref{Dixon-a}, we can catch hold of the
following equation:
\begin{align*}
&\sum_{k=0}^{(2n+1)/3}\frac{1-q^{6k-1}}{1-q^{-1}}\frac{(q^{-1},q^{-1+(3j+1)m},q^{-1-(3j+1)m};q^3)_k}
{(q^{3},q^{3-(3j+1)m},q^{3+(3j+1)m};q^3)_k}q^{3k}=0,
\end{align*}
where $0\leq j\leq (2n-m)/3m$. Noticing that these polynomials
$(1-aq^{m}),(1-aq^{4m}),\ldots,(1-aq^{2n}),(a-q^{m}),(a-q^{4m}),\ldots,(a-q^{2n})$
 are pairwise coprime, we obtain the following
$q$-supercongruence: modulo
$\prod_{j=0}^{(2n-m)/3m}(1-aq^{(3j+1)m})(a-q^{(3j+1)m}),$
\begin{align}
&\sum_{k=0}^{(2n+1)/3}[6k-1]\frac{(q^{-1},aq^{-1},q^{-1}/a;q^3)_k}{(q^3,q^3/a,aq^{3};q^3)_k}q^{3k}
\equiv0. \label{eq:wei-hh}
\end{align}
Guo and Schlosser \cite[Lemma 2.2]{GS20} showed that
\eqref{eq:wei-hh} is valid modulo $[n]$. Since the two polynomials
$\prod_{j=0}^{(2n-m)/3m}(1-aq^{(3j+1)m})(a-q^{(3j+1)m}))$ and $[n]$
are relatively prime, we get \eqref{eq:wei-gg} to complete the
proof.
\end{proof}

Similar to the proof of Lemma \ref{lemma-c}, we can derive the
following lemma.

\begin{lem}\label{lemma-f}
Let $n\geq7$ be an integer such that $n\equiv 1\pmod 6$ and let
$m>1$ be an integer such that $m|n$ and $m\equiv1\pmod{6}$. Then,
modulo $[n]\prod_{j=0}^{(2n-2m)/3m}(1-aq^{(3j+2)m})(a-q^{(3j+2)m}),$
\begin{align}
&\sum_{k=0}^{(2n+1)/3}[6k-1]\frac{(q^{-1},aq^{-1},q^{-1}/a;q^3)_k}{(q^3,q^3/a,aq^{3};q^3)_k}q^{3k}
\equiv0. \label{eq:wei-ii}
\end{align}
\end{lem}

Now we start to prove Theorem \ref{thm-c}.

\begin{proof}[Proof of Theorem \ref{thm-c}]
Considering $m|n$ and $m\equiv5\pmod{6}$, the $a=1$ case of the
expression
$[n]\prod_{j=0}^{(2n-m)/3m}(1-aq^{(3j+1)m})(a-q^{(3j+1)m})$ has the
factor
 $$\prod_{m|n, m\equiv5{\!\!\!\!\!}\pmod{6}}\Phi_m(q)^{(4n+7m)/3m},$$
 while the $a=1$ case of the expression $(q^3/a,aq^{3};q^3)_{(2n+1)/3}$,
 which is in the denominator of the left hand side of
 \eqref{eq:wei-aa}, only implies the factor
  $$\prod_{m|n, m\equiv5{\!\!\!\!\!}\pmod{6}}\Phi_m(q)^{(4n-2m)/3m}.$$

 Considering  $m|n$, $m>1$, and $m\equiv1\pmod{6}$, the $a=1$ case of the
product $[n]\prod_{j=0}^{(2n-2m)/3m}(1-aq^{(3j+2)m})(a-q^{(3j+2)m})$
has the factor
 $$\prod_{m|n, m>1,m\equiv1{\!\!\!\!\!}\pmod{6}}\Phi_m(q)^{(4n+5m)/3m},$$
 whereas the $a=1$ case of the $q$-shifted factorials $(q^3/a,aq^{3};q^3)_{(2n+1)/3}$,
 which emerges in the denominator of the left hand side of
 \eqref{eq:wei-cc}, merely implies the factor
  $$\prod_{m|n,m>1, m\equiv1{\!\!\!\!\!}\pmod{6}}\Phi_m(q)^{(4n-4m)/3m}.$$
Thus \eqref{eq:wei-c} is true modulo
$$\prod_{m|n, m\equiv5{\!\!\!\!\!}\pmod{6}}\Phi_m(q)^{3}\prod_{m|n, m>1,m\equiv1{\!\!\!\!\!}\pmod{6}}\Phi_m(q)^{3}
=\prod_{m|n, m>1,}\Phi_m(q)^{3}=[n]^3.$$
\end{proof}

{\bf{Acknowledgments}}\\

The work is supported by the National Natural Science Foundation of
China (No. 12571349).


\end{document}